\documentclass[12pt]{article}
\usepackage{graphicx}
\usepackage{amsmath,amsthm,amssymb,enumerate}
\usepackage{euscript,mathrsfs}
\usepackage{color}
\usepackage{dsfont}
\usepackage[left=2cm,right=2cm,top=3.5cm,bottom=3.5cm]{geometry}
\usepackage{color}
\usepackage[framemethod=tikz]{mdframed}
\allowdisplaybreaks

\usepackage{soul}

\catcode`\@=11 \@addtoreset{equation}{section}

\catcode`\@=12

\newtheorem{Theorem}{Theorem}[section]
\newtheorem{Proposition}[Theorem]{Proposition}
\newtheorem{Lemma}[Theorem]{Lemma}
\newtheorem{Corollary}[Theorem]{Corollary}

\theoremstyle{definition}
\newtheorem{Definition}[Theorem]{Definition}

\newtheorem{Remark}[Theorem]{Remark}

\newcommand{\bTheorem}[1]{
	\begin{Theorem} \label{T#1} }
	\newcommand{\eT}{\end{Theorem}}

\newcommand{\bProposition}[1]{
	\begin{Proposition} \label{P#1}}
	\newcommand{\eP}{\end{Proposition}}

\newcommand{\bLemma}[1]{
	\begin{Lemma} \label{L#1} }
	\newcommand{\eL}{\end{Lemma}}

\newcommand{\bCorollary}[1]{
	\begin{Corollary} \label{C#1} }
	\newcommand{\eC}{\end{Corollary}}

\newcommand{\bRemark}[1]{
	\begin{Remark} \label{R#1} }
	\newcommand{\eR}{\end{Remark}}

\newcommand{\bDefinition}[1]{
	\begin{Definition} \label{D#1} }
	\newcommand{\eD}{\end{Definition}}

\newcommand{\Del}{\Delta_x}

\newcommand{\wbfphi}{\widetilde{\bfphi}}

\newcommand{\bfphi}{\boldsymbol{\varphi}}

\newcommand{\bFormula}[1]{
	\begin{equation} \label{#1}}
	\newcommand{\eF}{\end{equation}}

\newcommand{\Ov}[1]{\overline{#1}}

\newcommand{\DC}{C^\infty_c}

\newcommand{\aleq}{\lesssim}

\newcommand{\toM}{\stackrel{\mathcal{M}_2}{\to}}
\newcommand{\toH}{\stackrel{\mathcal{H}}{\to}}

\newcommand{\vr}{\varrho}
\newcommand{\vre}{\vr_\ep}

\newcommand{\vue}{\vu_\ep}

\newcommand{\vu}{\vc{u}}

\newcommand{\vc}[1]{{\bf #1}}
\newcommand{\Ome}{\Omega_\ep}

\newcommand{\Div}{{\rm div}_x}
\newcommand{\Grad}{\nabla_x}

\newcommand{\dx}{\,{\rm d} {x}}

\newcommand{\dt}{\,{\rm d} t }

\newcommand{\intD}[1]{\int_D #1 \ \dx} 
\newcommand{\intO}[1]{\int_{\Omega} #1 \ \dx}

\newcommand{\intOe}[1]{\int_{\Omega_\ep} #1 \ \dx}

\newcommand{\D}{{\rm d}}

\newcommand{\ep}{\varepsilon}

\newcommand{\br}{ \nonumber \\ }

\def\softd{{\leavevmode\setbox1=\hbox{d}%
		\hbox to 1.05\wd1{d\kern-0.4ex{\char039}\hss}}}
\definecolor{Cgrey}{rgb}{0.85,0.85,0.85}
\definecolor{Cblue}{rgb}{0.50,0.85,0.85}
\definecolor{Cred}{rgb}{1,0,0}
\definecolor{fancy}{rgb}{0.10,0.85,0.10}

\newcommand\Cbox[2]{%
	\newbox\contentbox%
	\newbox\bkgdbox%
	\setbox\contentbox\hbox to \hsize{%
		\vtop{
			\kern\columnsep
			\hbox to \hsize{%
				\kern\columnsep%
				\advance\hsize by -2\columnsep%
				\setlength{\textwidth}{\hsize}%
				\vbox{
					\parskip=\baselineskip
					\parindent=0bp
					#2
				}%
				\kern\columnsep%
			}%
			\kern\columnsep%
		}%
	}%
	\setbox\bkgdbox\vbox{
		\color{#1}
		\hrule width  \wd\contentbox %
		height \ht\contentbox %
		depth  \dp\contentbox
		\color{black}
	}%
	\wd\bkgdbox=0bp%
	\vbox{\hbox to \hsize{\box\bkgdbox\box\contentbox}}%
	\vskip\baselineskip%
}

\mdfdefinestyle{MyFrame}{%
	linecolor=black,
	outerlinewidth=1pt,
	roundcorner=5pt,
	innertopmargin=\baselineskip,
	innerbottommargin=\baselineskip,
	innerrightmargin=10pt,
	innerleftmargin=10pt,
	backgroundcolor=white!20!white}

\begin{document}


\title{$\Gamma$--convergence for nearly incompressible fluids}

\date{}

\author{Peter Bella$^{1,}$\thanks{The work of P.B. was partially supported by the German Science Foundation DFG in context of the Emmy Noether Junior Research Group BE 5922/1-1.}
	\and Eduard Feireisl$^{2,}$\thanks{The work of E.F. was partially supported by the
		Czech Sciences Foundation (GA\v CR), Grant Agreement
		21--02411S. The Institute of Mathematics of the Academy of Sciences of
		the Czech Republic is supported by RVO:67985840. } \and Florian Oschmann$^{2,}$\thanks{The work of F.O. in Germany was partially supported by the German Science Foundation DFG in context of the Emmy Noether Junior Research Group BE 5922/1-1. After moving to the Czech Republic, the work was partially supported by the
		Czech Sciences Foundation (GA\v CR), Grant Agreement 22-01591S.}
}

\date{}

\maketitle

\medskip

\centerline{$^1$ TU Dortmund, Fakultät für Mathematik}

\centerline{Vogelpothsweg 87, 44227 Dortmund, Germany}

\medskip

\centerline{$^2$ Institute of Mathematics of the Academy of Sciences of the Czech Republic}

\centerline{\v Zitn\' a 25, CZ-115 67 Praha 1, Czech Republic}

\begin{abstract}
	
We consider the time-dependent compressible Navier--Stokes equations  in the low Mach number regime in a family of domains $\Omega_\ep\subset R^d$ converging in the sense of Mosco to a domain $\Omega \subset R^d$, $d \in \{2,3\}$. We show the limit is the incompressible Navier--Stokes system in $\Omega$.

\end{abstract}




\section{Introduction}
\label{P}

We consider the flow of a compressible viscous fluid in the low Mach number regime under domain perturbation. Our goal is to identify the the largest class possible of domain perturbations that have no impact on the incompressible limit. 

Specifically,  
let $(\Omega_\ep)_{\ep > 0}$ be a family of domains confined to a bounded set (ball) $D$, 
\begin{equation} \label{P1}
	\Omega_\ep \subset D \subset R^d,\ d \in \{2,3\} \ \mbox{for any}\ \ep > 0.
	\end{equation}
Along with $(\Omega_\ep)_{\ep > 0}$, we consider a family of solutions $(\vre, \vue)_{\ep > 0}$ of the barotropic Navier--Stokes system in the low Mach number regime: 
\begin{align} 
\partial_t \vre + \Div (\vre \vue) &= 0 	
	\label{P2} \\
\partial_t (\vre \vue) + \Div (\vre \vue \otimes \vue) + \frac{1}{{\rm Ma}^2(\ep)} \Grad p(\vre) &= \Div \mathbb{S}(\Grad \vue) + \vre \vc{f}	\label{P3}\\ 
\mathbb{S} (\Grad \vue) &= \mu \left( \Grad \vue + \Grad^t \vue - \frac{2}{d} \Div \vue \mathbb{I} \right) + \eta \Div \vue \mathbb{I} 
\label{P4}	
\end{align}
in $(0,T) \times \Omega_\ep$, supplemented with the no--slip boundary condition 
\begin{equation} \label{P5}
	\vue|_{\partial \Omega_\ep} = 0.
\end{equation}	
We suppose 
\begin{equation} \label{P6}
	\mu > 0,\ \eta \geq 0,\ p'(\vr) > 0 \ \mbox{for}\ \vr > 0,\ 
	p(\vr) \approx \vr^\gamma \ \mbox{for}\ \vr \to \infty.
\end{equation}

Intuitively, the minimal requirements on the domain convergence should guarantee stability of the associated Stokes problem. A convenient way to express this is Mosco's convergence.

\subsection{Mosco's convergence of fluid domains} 

The asymptotic behaviour of the family $(\Omega_\ep)_{\ep > 0}$ is described in terms of Mosco's convergence -- 
$\Gamma$--convergence associated to the Stokes operator. 

\begin{Definition}[\bf Convergence in the sense of Mosco] \label{DP1}

We say that $\Omega_\ep$ \emph{converge to a domain $\Omega \subset D$ in the sense of Mosco}, $\Omega_\ep \toM \Omega$, if: 

\begin{itemize}
	\item {\bf(M1)}  For any
	\[
	\bfphi_\ep \in W^{1,2}_0(\Omega_\ep; R^d) \subset W^{1,2}_0 (D; R^d), \ \bfphi_\ep \to \bfphi \ \mbox{weakly in}\ W^{1,2}_0 (D; R^d) 
	\]
there holds
\begin{equation} \label{P7}
\bfphi \in W^{1,2}_0 (\Omega; R^d).
\end{equation}

\item {\bf(M2)} For any $\bfphi \in C^1_c (\Omega; R^d)$, $\Div \bfphi = 0$, there exists a sequence 
$(\bfphi_\ep)_{\ep > 0}$,
\begin{equation} \label{P8}
\bfphi_\ep \in W^{1,2}_0(\Omega_\ep; R^d),\ \Div \bfphi_\ep = 0,\ 
\bfphi_\ep \to \bfphi \ \mbox{strongly in}\ W^{1,2}_0(D; R^d).
\end{equation}

\item {\bf(M2bis)} In addition to \eqref{P8}, the approximate sequence satisfies
\begin{align} 
	\bfphi_\ep \in W^{1,r}_0 (\Omega_\ep; R^d) 
	\ \mbox{for some}\ r > \frac{d \gamma}{2 \gamma - d},\ \gamma > \frac{d}{2}, \label{P9} \\ 
\limsup_{\ep \to 0} h(\ep) \| \bfphi_\ep \|_{W^{1,r}_0(\Omega_\ep; R^d) } < \infty 
\ \mbox{for some}\ h(\ep) \to 0 \ \mbox{as}\ \ep \to 0. \label{P10}
	\end{align}	
 	
	\end{itemize}

\end{Definition}

The conditions {\bf(M1)}, {\bf(M2)} represent the standard definition of Mosco's convergence. The extra hypotheses \eqref{P9}, \eqref{P10} assert that the $W^{1,r}$ can ``blow up'' at a rate 
$h(\ep)^{-1}$ independent of the approximated function $\bfphi$. Note that \eqref{P9}, \eqref{P10} follow from \eqref{P7}, \eqref{P8} (with $r=2$) if $\gamma>4d-6$, i.e., $\gamma > 2$ if $d=2$, and $\gamma > 6$ if $d=3$.

\subsection{Weak solutions to the Navier--Stokes system}

Next, we introduce the concept of weak solutions to the Navier--Stokes system. 

\begin{Definition}[\bf Weak solution] \label{PD2}
	
We say that $(\vre, \vue)$ is a \emph{weak solution} to the Navier--Stokes system \eqref{P2}--\eqref{P5} with the initial data
\begin{equation} \label{P11}	
\vre(0, \cdot) = \vr_{0, \ep},\ \vre \vue (0, \cdot) = (\vr \vu)_{0,\ep} 
\end{equation}
if the following holds:
\begin{itemize}
\item {\bf Integrability.}
\begin{align}
\vre \geq 0 \mbox{ a.a.~in } \Omega_\ep ,\ \vre &\in L^\infty(0,T; L^\gamma (\Omega_\ep)), \br
\vue &\in L^2(0,T; W^{1,2}_0 (\Omega_\ep; R^d)),\ \vre \vue \in L^\infty(0,T; L^{\frac{2 \gamma}{\gamma + 1}}(\Omega_\ep; R^d)).
		\label{P12}
\end{align}

\item {\bf Equation of continuity.}

\begin{equation} \label{P13}
\int_0^T \intOe{ \left[ \vre \partial_t \varphi + \vre \vue \cdot \Grad \varphi \right] } \dt = - \intOe{ \vr_{0, \ep} \varphi (0,\cdot) }
\end{equation}
for any $\varphi \in C^1_c([0,T) \times \Ov{\Omega}_\ep; R^d)$.

\item {\bf Momentum equation.}

\begin{align}
\int_0^T &\intOe{ \left[ \vre \vue \cdot \partial_t \bfphi + \vre \vue \otimes \vue : \Grad \bfphi + \frac{1}{{\rm Ma}^2(\ep)} p(\vre) \Div \bfphi \right] } \dt \br &=  
\int_0^T \intOe{ \left[ \mathbb{S}(\Grad \vue) : \Grad \bfphi - \vre \vc{f} \cdot \bfphi \right] } \dt - \intOe{ (\vr \vu)_{0,\ep} \cdot \bfphi(0, \cdot) }  \label{P14}
\end{align}
for any $\bfphi \in C^1_c([0,T) \times \Ome; R^d)$.

\item {\bf Energy inequality.}

\begin{align} 
&\intOe{ \left[ \frac{1}{2} \mathds{1}_{\{\vre > 0\}} \frac{ |\vre \vue|^2 }{\vre} + \frac{1}{{\rm Ma}^2(\ep) } P(\vre) \right] (\tau, \cdot) } + \int_0^\tau \intOe{ \mathbb{S}(\Grad \vue) : \Grad \vue } \dt 
\br &\quad \leq \intOe{ \left[  \frac{1}{2} \frac{ |(\vr \vu)_{0, \ep} |^2 }{\vr_{0, \ep}} + \frac{1}{{\rm Ma}^2(\ep) } P(\vr_{0,\ep}) \right] } + 
\int_0^\tau \intOe{ \vre \vc{f} \cdot \vue } \dt \label{P15}
\end{align}
for any $\tau \in (0,T]$, where 
\[
P'(\vr) \vr - P(\vr) = p(\vr).
\]

\end{itemize}
\end{Definition}

The existence of weak solutions for $p(\vr) \approx \vr^\gamma$, $\gamma > \frac{d}{2}$ considered on a general spatial domain was proved in \cite{FNP1}.

\subsection{Main result}

\subsubsection{Well--prepared initial data}

We consider \emph{well--prepared} initial data. Specifically, we suppose that there is a constant $\Ov{\vr}$ such that 
\begin{equation}  \label{P16}
	\frac{1}{{\rm Ma}^2(\ep) }\intOe{ \Big[ P(\vr_{0,\ep}) - P'(\Ov{\vr})(\vr_{0,\ep} - \Ov{\vr} ) - P(\Ov{\vr}) \Big]  } \to 0 
	\ \mbox{as}\ \ep \to 0.
	\end{equation}
Moreover, we suppose there is a solenoidal velocity field 
\begin{equation} \label{P17}
\vu_0 \in W^{1,p}(\Omega; R^d) \ \mbox{for any}\ 1 \leq p < \infty,\ \Div \vu_0 = 0,\ \vu_0|_{\partial \Omega} = 0,
\end{equation}
such that 
\begin{equation} \label{P18}
(\vr \vu)_{0, \ep} \to \Ov{\vr} \vu_0 \ \mbox{in}\ L^2(D; R^d),\ 
\intOe{  \frac{ |(\vr \vu)_{0, \ep} |^2 }{\vr_{0, \ep}} } \to \intO{ \Ov{\vr} |\vu_0 |^2 } \ \mbox{as}\ \ep \to 0.
\end{equation}	

\subsubsection{Convergence}

Having collected the necessary preparatory material, we are ready to state our main result.

\begin{mdframed}[style=MyFrame]
	
\begin{Theorem}[{\bf Asymptotic limit}] \label{TP1}	
	
Let $(\Ome)_{\ep > 0}$ be a family of bounded domains, 
\begin{equation} \label{P19}	
\Ome \subset D \subset R^d, \ d \in \{2,3\},\ \Ome \toM \Omega, 
\end{equation}
where $\Omega \subset R^d$ is a bounded domain of class $C^{2 + \nu}$. 
Let $(\vre, \vue)_{\ep > 0}$ be a family of weak solutions to the Navier--Stokes system \eqref{P2}--\eqref{P5} in $(0,T) \times \Ome$ 
specified in Definition \ref{PD2} emanating from the well--prepared initial data \eqref{P11}, where 
\begin{align} 
\vc{f} &\in L^\infty((0,T) \times D; R^d),\ \label{P20} \\
p &\in C[0,\infty) \cap C^1(0, \infty),\ p' (\vr) > 0 \ \mbox{for}\ \vr > 0,\ 
0 < \underline{p} \leq \liminf_{\vr \to \infty} \frac{p(\vr)}{\vr^\gamma } \leq   \limsup_{\vr \to \infty} \frac{p(\vr)}{\vr^\gamma } 
\leq \Ov{p}.
\label{P21}
\end{align}
Finally, let 
\begin{equation} \label{P22}
\gamma > \frac{d}{2},\ \limsup_{\ep \to 0} {\rm Ma}^{\frac{2}{\gamma}} (\ep) h^{-1}(\ep) < \infty,
	\end{equation}
where $h(\ep)$ has been introduced in \eqref{P10}.

Then there exists $0 < T_{\rm max} \leq \infty$, $T_{\rm max} = \infty$ if $d=2$, such that
\begin{align} 
\vre \to \Ov{\vr} \ &\mbox{in} \ L^\infty(0, T; L^\gamma(D))   , \label{P23} \\  \vue \to \vu \ &\mbox{in}\ L^2((0,T) \times D; R^d)
\ \mbox{and weakly in}\ L^2(0, T; W^{1,2}_0(D; R^d))
\label{P24}	
	\end{align}
for any $0 < T < T_{\rm max}$, where $\vu$ is the (unique) strong solution of the incompressible Navier--Stokes system 
\begin{align} 
	\Div \vu &= 0, \br
\Ov{\vr} \partial_t \vu + \Ov{\vr} \vu \cdot \Grad \vu + \Grad \Pi &= \Div \mathbb{S}(\Grad \vu) + \Ov{\vr} \vc{f}, \br 
\vu|_{\partial \Omega} &= 0, \label{P25}	
	\end{align}
in $[0, T_{\rm max}) \times \Omega$, starting from the initial datum 
\begin{equation} \label{P26}
	\vu(0, \cdot) = \vu_0.
\end{equation}

\end{Theorem}	
	\end{mdframed}

\begin{Remark} \label{PR1}
	In \eqref{P23}, \eqref{P24},
	we have extended $\vre = \Ov{\vr}$, $\vue = 0$ outside $\Ome$.
\end{Remark}

The rest of the paper is devoted to the proof of Theorem \ref{TP1}. In Section \ref{B}, we derive the necessary uniform bounds independent of $\ep \to 0$. In Section \ref{C}, we show convergence to the target system.
In Section \ref{A}, we present some specific examples of the family $(\Ome)_{\ep > 0}$, including homogenization results for the optimal range of the adiabatic exponent $\gamma>\frac{d}{2}$ in the spirit of \cite{DieFeLu, LuSchw, NecOsch}, where no low Mach number limit is needed for large values of $\gamma>4d-6$, as well as domain stability of the governing equations as investigated in \cite{FeKaKrSt, FNP1, PLSO1}.

\section{Uniform bounds}
\label{B}

It is convenient to extend
\[
\vue = 0,\ \vre = \Ov{\vr}\ \mbox{in}\ D \setminus \Ome
\]
to work with functions defined on a single domain. Accordingly, the energy inequality \eqref{P15} reads
\begin{align} 
	&\intD{ \left[ \frac{1}{2} \mathds{1}_{\{\vre > 0\}} \frac{ |\vre \vue|^2 }{\vre}  + \frac{1}{{\rm Ma}^2(\ep) }\left(  P(\vre) - P'(\Ov{\vr}) (\vre - \Ov{\vr} )- P(\Ov{\vr}) \right) \right] (\tau, \cdot) } \br &\quad + \int_0^\tau \intD{ \mathbb{S}(\Grad \vue) : \Grad \vue } \dt 
	\br &\quad \leq \intOe{ \left[  \frac{1}{2} \frac{ |(\vr \vu)_{0, \ep} |^2 }{\vr_{0, \ep}} + \frac{1}{{\rm Ma}^2(\ep) } \left( P(\vr_{0,\ep}) - P'(\Ov{\vr}) (\vr_{0,\ep} - \Ov{\vr} ) -  P(\Ov{\vr}) \right) \right] } + 
	\int_0^\tau \intD{ \vre \vc{f} \cdot \vue } \dt \br  
&\quad \lesssim \intOe{ \left[  \frac{1}{2} \frac{ |(\vr \vu)_{0, \ep} |^2 }{\vr_{0, \ep}} + \frac{1}{{\rm Ma}^2(\ep) } \left( P(\vr_{0,\ep}) - P'(\Ov{\vr}) (\vr_{0,\ep} - \Ov{\vr} ) -  P(\Ov{\vr}) \right) \right] } \br &\quad + 
 \int_0^\tau \intD{ \vre  } \dt +  \int_0^\tau \intD{ \vre |\vue |^2 } \dt,	
	\label{B1}
\end{align}
where we have used hypothesis \eqref{P20}.

As $h(\ep) \to 0$, hypothesis \eqref{P22} yields ${\rm Ma}(\ep) \to 0$. Applying Gr\" onwall's lemma along with hypotheses \eqref{P16}, \eqref{P18} we get the estimate 
\begin{align} 
\sup_{\tau \in (0,T) }	&\intD{ \left[ \frac{1}{2} \mathds{1}_{\{\vre > 0\}} \frac{ |\vre \vue|^2 }{\vre}  + \frac{1}{{\rm Ma}^2(\ep) }\left(  P(\vre) - P'(\Ov{\vr}) (\vre - \Ov{\vr} )- P(\Ov{\vr}) \right) \right] (\tau, \cdot) }  \aleq 1, \br &\int_0^T \intD{ \mathbb{S}(\Grad \vue) : \Grad \vue } \dt \aleq 1
	\label{B2}
\end{align}
uniformly for $\ep \to 0$. As ${\rm Ma}(\ep) \to 0$ and $p$ complies with the growth condition 
\eqref{P21}, this implies \eqref{P23}, more specifically,
\begin{equation} \label{B3}
\vre \to \Ov{\vr} \ \mbox{in} \ L^\infty(0,T; L^\gamma (D)).
\end{equation}

In addition, applying Korn--Poincar\' e inequality, we get 
\begin{equation} \label{B4}
	\vue \to \vu \ \mbox{weakly in}\ L^2(0,T; W^{1,2}_0 (D; R^d)), 
	\end{equation}
passing to a suitable subsequence as the case may be. 

Finally, we claim that 
\begin{equation} \label{B5}
	\vu(t, \cdot) \in W^{1,2}_0(\Omega; R^d) \ \mbox{for a.a.}\ t \in (0,T). 
	\end{equation}
Indeed, consider a suitable time regularization $\eta_\delta * \vue$, where $\eta_\delta = \eta_\delta(t)$ is 
a family of regularizing kernels. In accordance with hypothesis {\bf (M1)} we get 
\[
\eta_\delta * \vue \to \eta_\delta * \vu \ \mbox{weakly in}\ 
W^{1,2}_0(D; R^d) \ \mbox{for any}\ t \in (0,T),\ \mbox{where } \eta_\delta * \vu \in W^{1,2}_0(\Omega; R^d) 
\]
for any $\delta > 0$. Letting $\delta \to 0$ yields \eqref{B5}.

\section{Convergence}
\label{C}

To complete the proof of Theorem \ref{TP1}, it is enough to perform the limit in \eqref{P13}--\eqref{P15}. 

\subsection{Asymptotic limit in the equation of continuity}

First, it follows from \eqref{B2} and \eqref{B3} that 
\begin{equation} \label{C1}
	\vre \vue \to \Ov{\vr} \vu \ \mbox{weakly-(*) in}\ L^\infty(0,T; L^{\frac{2 \gamma}{\gamma + 1}}(D; R^d)).
	\end{equation}
In particular, letting $\ep \to 0$ in the equation of continuity \eqref{P13} yields 
\begin{equation} \label{C2}
	\Div \vu = 0 \ \mbox{a.a. in}\ (0,T) \times \Omega.
	\end{equation}

\subsection{Asymptotic limit in the momentum equation}

Next, we perform the limit in the momentum equation \eqref{P14}. Let $B \in C^\infty [0, \infty)$ be a cut--off function such that
\begin{equation} \label{C3}
0 \leq	B(\vr) \leq 1,\ B(\vr) = \begin{cases} 1 & \mbox{if}\ 0 \leq \vr \leq \Ov{\vr} + 1, \\ 
	\in [0,1] & \mbox{if}\ \Ov{\vr} + 1 \leq \vr \leq \Ov{\vr} + 2, \\ 
	0 & \mbox{if}\ \vr \geq \Ov{\vr} + 2.
	\end{cases}
\end{equation}
We set 
\[
\vr_{\ep, B} = B(\vre) \vre ,\ \vr_{\ep,B}^c = (1 - B(\vre)) \vre. 
\]
Since $\vr_{\ep, B}$ is bounded by $\Ov{\vr}+2$, we have in accordance with \eqref{B3} 
\begin{equation} \label{C4} 
	\vr_{\ep,B}  \to \Ov{\vr} \ \mbox{in}\ L^\infty (0,T; L^q (D)) \ \mbox{for any}\ 
	1 \leq q < \infty.
\end{equation}	 
Moreover, by virtue of hypothesis \eqref{P21} and the uniform bounds \eqref{B2}, 
\begin{equation} \label{C5}
\intD{ |\vr_{\ep,B}^c |^\gamma } \aleq \intD{ \Big( P(\vre) - P' (\Ov{\vr})(\vre - \Ov{\vr} ) - P(\Ov{\vr}) \Big) 
} \aleq {\rm Ma}^2 (\ep) \ \mbox{uniformly for}\ \ep > 0.	
	\end{equation}

We are ready to perform the limit in the momentum equation. We consider a specific ansatz for the test functions, 
\[
\psi(t) \bfphi(x),\ \psi \in C^\infty_c [0,T), \ \bfphi \in C^\infty_c (\Omega; R^d),\ \Div \bfphi = 0.
\]
By virtue of hypothesis \eqref{P19}, there is an approximate sequence $(\bfphi_\ep)_{\ep > 0}$ enjoying the properties \eqref{P8}--\eqref{P10}. Accordingly, the quantities 
\[
(\psi (t) \bfphi_\ep (x) )_{\ep > 0} 
\]
are eligible test functions for \eqref{P14}. Using \eqref{P8}, hypotheses \eqref{P18}, \eqref{P20}, and the convergence \eqref{B4}, we easily obtain
\begin{align} 
	\int_0^T \intOe{ \vre \vc{f} \cdot (\psi \bfphi_\ep) } \dt &\to 
	\int_0^T \intO{ \Ov{\vr} \vc{f} \cdot (\psi \bfphi) } \dt, \br 
	\int_0^T \intOe{ \psi \mathbb{S}( \Grad \vue) : \Grad \bfphi_\ep } \dt &\to  
	\int_0^T \intO{ \mathbb{S}(\Grad \vu) : \Grad (\psi \bfphi) } \dt
\label{C6}	
	\end{align}
as $\ep \to 0$.

Next, rewrite the convective term in the form 
\begin{align}
\int_0^T &\intOe{ \psi \vre \vue \otimes \vue : \Grad \bfphi } \dt \br &= 
\int_0^T \intOe{ \psi \vr_{\ep,B} \vue \otimes \vue : \Grad \bfphi_\ep } \dt
+ \int_0^T \intOe{ \psi \vr_{\ep,B}^c \vue \otimes \vue : \Grad \bfphi_\ep } \dt.
\nonumber
\end{align}
Suppose $d=3$.
By H\"older's inequality and the Sobolev embedding relation $W^{1,2} \hookrightarrow L^6$, we get 
\begin{align}
	\left| \intOe{ \vr^c_{\ep, B} \vue \otimes \vue : \Grad \bfphi_\ep } \right|
	&\aleq \| \vr^c_{\ep,B} \|_{L^\gamma (\Ome)} \| \vue \|_{W^{1,2}_0 (\Ome; R^d)}^2 \| \Grad \bfphi_\ep \|_{L^q(\Ome; R^{d\times d} ) } \br &\mbox{with}\ q = \frac{3 \gamma}{2 \gamma - 3}.
\nonumber	
	\end{align}
Consequently, using hypothesis \eqref{P10} and the bound \eqref{C5}, we conclude 
\begin{align}
	\left| \intOe{ \vr^c_{\ep, B} \vue \otimes \vue : \Grad \bfphi_\ep } \right|
	&\aleq {\rm Ma}^{\frac{2}{\gamma} }(\ep) h^{-1}(\ep) \| \vue \|_{W^{1,2}_0 (\Ome; R^d)}^2 
	h(\ep) \| \Grad \bfphi_\ep \|_{L^q(\Ome; R^{d\times d} ) } \br &\mbox{with}\ q = \frac{3 \gamma}{2 \gamma - 3},
	\label{C7}	
\end{align}
where, by virtue of \eqref{P8}--\eqref{P10} and a simple interpolation argument,
\begin{equation} \label{C8}
h(\ep) \| \Grad \bfphi_\ep \|_{L^q(\Ome; R^{d\times d} ) } \to 0 \ \mbox{as}\ \ep \to 0.	
	\end{equation}
Repeating a similar argument for $d=2$ we conclude
\begin{equation} \label{C9} 
\int_0^T \intOe{ \psi \vr_{\ep,B}^c \vue \otimes \vue : \Grad \bfphi_\ep } \dt \to 0.	
	\end{equation}
Finally, we use \eqref{C4} to deduce 
\begin{equation} \label{C10} 
\int_0^T \intOe{ \psi \vr_{\ep,B} \vue \otimes \vue : \Grad \bfphi_\ep } \dt \to 
\int_0^T \intO{ \Ov{ \vr \vu \otimes \vu } : \Grad (\psi \bfphi) } \dt,	
	\end{equation}
where 
\begin{equation} \label{C11}
\vr_{\ep, B} \vue \otimes \vue \to \Ov{\vr \vu \otimes \vu} \ \mbox{weakly in}\ L^s(0,T; L^2(D; R^{d \times d})) 
\ \mbox{for some}\ s > 1.	
	\end{equation}
Indeed, the quantity $\vr_{\ep, B} \vue \otimes \vue$ is bounded in $L^\infty(0,T; L^1(D; R^{d \times d}))$ and also
in $L^1(0,T; L^3 (D; R^{d \times d}))$; hence the desired result follows by interpolation. 

Using similar arguments, we also get 
\[
\sqrt{\vre} \vue \to \sqrt{\Ov{\vr}} \vu \ \mbox{weakly-(*) in}\ L^\infty(0,T; L^2(D; R^{d \times d})).
\]
As the function 
\[
\vc{v} \mapsto \vc{v} \otimes \vc{v} : (\xi \otimes \xi),\ \xi \in R^d \ \mbox{is convex}
\]
for any fixed $\xi$, we conclude 
\begin{equation} \label{C12}
	\mathcal{R} \equiv  \Ov{\vr \vu \otimes \vu} - \Ov{\vr} \vu \otimes \vu \geq 0 
\ \mbox{in the sense of symmetric matrices.}
	\end{equation}

Finally, by virtue of similar arguments, it is easy to show
\begin{equation} \label{C13}
	\int_0^T \partial_t \psi \intOe{ \vre \vue \cdot \bfphi_\ep } \dt 
	\to 	\int_0^T \partial_t \psi \intO{ \Ov{\vr} \vu \cdot \bfphi } \dt
	\ \mbox{as}\ \ep \to 0.
	\end{equation} 
Using the density of the products $\psi(t) \bfphi(x)$ we deduce the limit momentum balance:
\begin{align} 
	\int_0^T &\intO{ \Big[ \Ov{\vr} \vu \cdot \partial_t \bfphi + \Ov{\vr} \vu \otimes \vu : \Grad \bfphi \Big] } \dt \br
&= \int_0^T \intO{ \Big[ \mathbb{S}(\Grad \vu) : \Grad \bfphi - \Ov{\vr} \vc{f} \cdot \bfphi - 
	\mathcal{R} : \Grad \bfphi \Big] } \dt - 
\intO{ \Ov{\vr} \vu \cdot \bfphi(0, \cdot) } \label{C14}
\end{align}
for any $\bfphi \in C^1_c([0,T) \times \Omega; R^d)$, $\Div\bfphi = 0$.

\subsection{Asymptotic limit in the energy inequality}

Finally, we multiply the energy inequality \eqref{P15} on $\psi \in C_c(0,T)$, $\psi \geq 0$, 
$\int_0^T \psi \dt = 1$ and integrate obtaining 
\begin{align} 
\int_0^T \psi(\tau)	&\intOe{ \left[ \frac{1}{2} \vre |\vue |^2 + \frac{1}{{\rm Ma}^2(\ep) } P(\vre) \right] (\tau, \cdot) } 
\ \D \tau + \int_0^T \psi (\tau) \int_0^\tau \intOe{ \mathbb{S}(\Grad \vue) : \Grad \vue } \dt \ \D \tau
	\br &\quad \leq \intOe{ \left[  \frac{1}{2} \frac{ |(\vr \vu)_{0, \ep} |^2 }{\vr_{0, \ep}} + \frac{1}{{\rm Ma}^2(\ep) } P(\vr_{0,\ep}) \right] } + 
\int_0^T \psi (\tau)	\int_0^\tau \intOe{ \vre \vc{f} \cdot \vue } \dt \ \D \tau. \nonumber
\end{align}
Using hypotheses \eqref{P16}, \eqref{P18} we perform the limit $\ep \to 0$ to deduce
\begin{align} 
	\int_0^T \psi(\tau)	&\intO{ \left[ \frac{1}{2} \Ov{\vr} |\vu |^2 + \mathcal{C} \right] (\tau, \cdot) } 
	\ \D \tau + \int_0^T \psi (\tau) \int_0^\tau \intO{ \mathbb{S}(\Grad \vu) : \Grad \vu } \dt \ \D \tau
	\br &\quad \leq \intO{ \frac{1}{2} \Ov{\vr} |\vu_0|^2  } + 
	\int_0^T \psi (\tau)	\int_0^\tau \intO{ \Ov{\vr} \vc{f} \cdot \vu } \dt \ \D \tau, \label{C15}
\end{align}
where 
\begin{equation} \label{C16}
	\mathcal{C} = \frac{1}{2} \Ov{\vr |\vu|^2 } - \frac{1}{2} \Ov{\vr} |\vu|^2 \approx {\rm trace}[ \mathcal{R} ].
\end{equation}	
As the function $\psi$ is arbitrary, we may infer that
\begin{align} 
	&\intO{ \left[ \frac{1}{2} \Ov{\vr} |\vu |^2 + \mathcal{C} \right] (\tau, \cdot) } 
+ \int_0^\tau \intO{ \mathbb{S}(\Grad \vu) : \Grad \vu } \dt 
	\br &\quad \leq \intO{ \frac{1}{2} \Ov{\vr} |\vu_0|^2  } + 
	\int_0^\tau \intO{ \Ov{\vr} \vc{f} \cdot \vu } \dt \ \mbox{for a.a.}\ \tau \in (0,T). \label{C17}
\end{align}	

\subsection{Application of the weak--strong uniqueness principle}

In view of \eqref{C2}, \eqref{C14}, and \eqref{C16}, \eqref{C17}, the limit $\vu$ is a dissipative solution of the 
incompressible Navier--Stokes system in the sense of \cite{AbbFei2}. Moreover, as $\vu_0$ as well as $\Omega$ are smooth enough, the incompressible Navier--Stokes system \eqref{P25}, \eqref{P26} admits a regular solution defined on a maximal interval $[0, T_{\rm max})$, where $T_{\rm max} = \infty$ if $d=2$. Consequently, the weak--strong uniqueness principle \cite[Theorem 2.8]{AbbFei2} completes the proof of Theorem \ref{TP1}.

\section{Applications}
\label{A}

We start by a short discussion on the hypotheses {\bf (M1), (M2), (M2bis)} stipulated by Mosco's convergence. As observed by Henrot and Pierre \cite[Section 3.5.2]{HenPie}, the conditions 
are ``almost necessary'' for Theorem \ref{TP1} to hold. The next result shows some necessary conditions for {\bf (M2)} to imply {\bf (M2bis)}.

\begin{Lemma} \label{LA1}
	
	Let $(\Ome)_{\ep > 0}$ be a family of domains satisfying: 
	\begin{itemize}
		
		\item 
		\[
		\Omega,\ \Ome \ \mbox{are of class}\ C^2 \ \mbox{for any}\ \ep > 0; 
		\]
		\item 
		
		\[
		\mbox{hypothesis {\bf (M2)} is satisfied.}
		\]
		
		\end{itemize}
	
	Then there exists a function $h(\ep)$ such that {\bf (M2bis)} holds.
	
	\end{Lemma}

\begin{proof}
	
	Given a function $\bfphi \in \DC(\Omega; R^d)$, $\Div \bfphi = 0$, we have to construct a sequence $(\bfphi_\ep)_{\ep > 0}$ satisfying \eqref{P8}--\eqref{P10}. 
	For each $\ep > 0$ we consider the minimization problem 
	\[
	I_\ep = {\rm inf}_{ \vc{h} \in W^{1,2}_0(\Ome; R^d), \Div \vc{h} = 0} \| \Grad {\bfphi} - \Grad \vc{h} \|^2_{L^2(D; R^{d \times d})} .
 \] 
	On the one hand, as {\bf (M2)} holds, $I_\ep \to 0$ as $\ep \to 0$. On the other hand, 
	\begin{equation} \label{A1a}
	\| \Grad {\bfphi} - \Grad \vc{h} \|_{L^2(D; R^{d \times d})}^2 = 
	\| \Grad {\bfphi} - \Grad \vc{h} \|_{L^2(\Ome; R^{d \times d})}^2 + \| \Grad \bfphi \|_{L^2(\Omega \setminus \Ome; R^{d \times d})}^2,  
	\end{equation}
	therefore
	\begin{equation} \label{A1}
	\| \Grad \bfphi \|_{L^2(\Omega \setminus \Ome; R^{d \times d})}^2 \to 0 \ \mbox{as}\ \ep \to 0.	
	\end{equation}

Consider
	\[
		\bfphi_{\ep} = {\rm arg min}_{ \vc{h} \in W^{1,2}_0(\Ome; R^d), \Div \vc{h} = 0}\| \Grad {\bfphi} - \Grad \vc{h} \|^2_{L^2(\Ome; R^{d \times d})}. 
		\]
It follows from \eqref{A1a}, \eqref{A1} that $(\bfphi_\ep)_{\ep > 0}$ satisfies {\bf (M2)}. Moreover, $\bfphi_\ep$
is the unique solution of the Stokes problem 
	\begin{equation} \label{A2}
		- \Del \bfphi_\ep + \Grad \Pi_\ep = -\Del \bfphi,\ \Div \bfphi_\ep = 0 \ \mbox{in}\ \Ome,\ \bfphi_\ep|_{\partial \Ome} = 0.
		\end{equation}
As $\Ome$ are regular domains, the standard $L^r-$estimates for the Stokes operator give rise to  
\begin{equation} \label{A3}
	\| \bfphi_\ep \|_{W^{1,r}(\Omega; R^d)} \leq C(\Ome, r) \| \bfphi \|_{C^1_c(\Omega; R^d)},\ 1 < r < \infty,
	\end{equation}	
which yields the desired conclusion {\bf (M2bis)} with $h(\ep) = C(\Ome, r)^{-1}$.	
	
	\end{proof}

\subsection{Convergence in the Hausdorff (complementary) topology}

Let $(\Ome)_{\ep > 0}$ be a family of domains in $R^d$, $d \in \{2,3\}$ contained in a fixed bounded domain $D$. We say that 
$\Omega_\ep \toH \Omega$ in the Hausdorff complementary topology if
\[
d_H [ D \setminus \Omega_\ep; D \setminus \Omega ] \to 0 \ \mbox{as}\ \ep \to 0,
\]
where
\[
d_H [K_1; K_2] = \sup \left\{ \sup_{x \in K_1} {\rm dist}[x; K_2] ; \sup_{x \in K_2} {\rm dist}[x, K_1]    \right\}. 
\]
As shown in \cite[Chapter 2, Corollary 2.2.26]{HenPie}, \emph{any} uniformly bounded sequence of domains contains a subsequence converging to a certain $\Omega$ in the 
Hausdorff topology.

The important property of the Hausdorff convergence is that any compact set in the limit domain $\Omega$ is ultimately contained in $\Omega_\ep$ for $\ep > 0$ small enough, see \cite[Chapter 2, Proposition~2.2.18]{HenPie}:
\begin{equation} \label{A5}
	\Ome \toH \Omega \ \Rightarrow \ \ \mbox{for any compact}\ K \subset \Omega,\ \mbox{there exists}\ \ep_K > 0 
	\ \mbox{such that}\ K \subset \Ome \ \mbox{for any}\ 0 < \ep < \ep_K.
\end{equation}
In particular, if $\Ome \toH \Omega$, then {\bf (M2)}, {\bf (M2bis)} are automatically satisfied for \emph{any} $h(\ep) \to 0$.
As for {\bf (M1)} and $d=2$, there is a very elegant criterion due to \v Sver\' ak \cite{Sver} (see also \cite[Chapter 3, Theorem~3.4.14 and Proposition~3.5.5]{HenPie}):
\begin{align} 
	\Ome \subset D \subset R^2, \ \Ome \toH \Omega,\ &\# \{ \ \mbox{connected components}\ D \setminus \Ome \} \leq k 
	\ \mbox{uniformly for}\ \ep \to 0 \br &\Rightarrow \ \mbox{{\bf (M1)} holds.} \label{A6}
	\end{align}
If $d=3$, similar criteria have been obtained by Bucur and Zolesio \cite{BZ} in terms of capacity.	

\subsection{Domains perforated by tiny holes}

The example is inspired by homogenization, where the uniform boundedness assumption in \eqref{A6} might not be satisfied anymore. We consider 
\[
\Ome = \Omega \setminus K_\ep ,\ K_\ep \subset \Omega \ \mbox{a compact set.}
\]
Obviously, the property {\bf (M1)} is always satisfied. We therefore focus on {\bf (M2), (M2bis)}. 

Let us start with $d = 3$. 

\begin{Proposition} \label{PA1}
	Let 
\[
\Ome = \Omega \setminus K_\ep ,\ K_\ep \subset \Omega \subset R^3. 
\]
Suppose there exists a family of balls $B_{r_i(\ep)}(x_i(\ep))$ of radii $r_i(\ep)$, $i=1,\dots, N(\ep)$ such that 
\begin{align} 
K_\ep &\subset \bigcup_{i = 1}^{N(\ep)} B_{r_i(\ep)}(x_i(\ep)), \br {\rm dist}[ x_i , \partial \Omega] &> 2 r_i(\ep) , 
\ {\rm dist}[ x_i, x_j ] > 2 (r_i(\ep) + r_j(\ep) ) \ \mbox{for}\ i \ne j, \br 
\sum_{i=1}^{N(\ep)} r_i(\ep) &\to 0 \ \mbox{as}\ \ep \to 0 . \label{A7}	
	\end{align}

Then {\bf (M2)} holds. Moreover, if 
\begin{equation} \label{A8}
	h(\ep) = \min_{1 \leq i \leq N(\ep)} r_i(\ep), 
	\end{equation}
then	 {\bf (M2bis)} is satisfied for any $\gamma > 3/2$, specifically, \eqref{P10} holds for any $r < \infty$.
	\end{Proposition}

\begin{proof}
	
	Let $\bfphi \in C^1_c(\Omega; R^d)$, $\Div \bfphi = 0$ be given. Consider a function 
\begin{align} 
	H &\in C^\infty(R), \ 0 \leq H(Z) \leq 1,\ H'(Z) = H'(1- Z) \ \mbox{for all}\ Z \in R, \br 
	H(Z) &= 0 \ \mbox{for} \ - \infty < Z \leq \frac{1}{4},\ 
	H(Z) = 1 \ \mbox{for}\ \frac{3}{4} \leq Z < \infty.
	\label{A9}
\end{align}
Set 
\begin{equation} \label{A10}
\vc{E}_{\ep,i}[\bfphi] =  \bfphi(x) H \left( \frac{ | x - x_i(\ep) |}{r_i(\ep)} - 1 \right), 
\end{equation}
and 
\begin{equation} \label{A11}
	\bfphi_{\ep,i} = \vc{E}_{\ep,i}[\bfphi] - \mathcal{B}_{\ep,i}[\Div \vc{E}_{\ep,i} [\bfphi]].
\end{equation}
Here $\mathcal{B}_{\ep,i}$ denotes the so--called Bogovski\u{\i} operator defined on the set 
\[
Q_{i,\ep} = B_{2 r_i(\ep)}(x_i(\ep)) \setminus B_{r_i(\ep)} (x_i(\ep))
\]
and enjoying the following properties: 
\begin{align} 
\mathcal{B}_{\ep,i} [h] \in W^{1,q}_0 (Q_{i,\ep}; R^d),\ \Div \mathcal{B}_{\ep,i} [h] = h \ \mbox{for any}\ h \in L^q(Q_{i,\ep}),\ \int_{Q_{i,\ep}} h \dx = 0,	\label{A12} \\
\| \Grad \mathcal{B}_{\ep,i} [h] \|_{L^q(Q_{i,\ep}; R^{d \times d})} \leq c(q) \| h \|_{L^q(Q_{i,\ep})} ,\ 1 < q < \infty,  \label{A13}
	\end{align}
where the constant in \eqref{A13} is independent of $i$ and $\ep$.
It is easy to check that
\begin{align} 
	\bfphi_{\ep, i} &= 0 \ \mbox{in}\ B_{r_i(\ep)}(x_i(\ep)) \ , \br 
	\bfphi_{\ep, i} &= \bfphi \ \mbox{in}\ \Omega \setminus B_{2 r_i(\ep)}(x_i(\ep)), \br 
	\Div \bfphi_{\ep, i} &= 0 \ \mbox{in}\ \Omega 
	\label{A14}
	\end{align} 
for any $i = 1,\dots, N(\ep)$. Note carefully that
\[
\int_{Q_{\ep,i}} \Div \vc{E}_{\ep, i}[\bfphi] = 0 
\]
so that the operator $\mathcal{B}_{\ep,i}$ is well defined. Thus, a suitable approximation of $\bfphi$ can be defined as
\begin{equation} \label{A15}
	\bfphi_\ep (x) = \left\{ \begin{array}{l} \bfphi_{\ep,i} (x) \ \mbox{if}\ x \in B_{2r_i(\ep)}(x_i(\ep)), \\ 
		\bfphi \ \mbox{otherwise.} \end{array} \right.
\end{equation}

Now, we compute the error 
\begin{equation} \label{A16} 
\| \Grad \bfphi - \Grad \bfphi_\ep \|_{L^q(\Omega; R^{d \times d})}^q \leq 
\sum_{i = 1}^{N(\ep)} \| \Grad \bfphi \|_{L^q(B_{2r_i(\ep)}(x_i(\ep)); R^{d \times d})}^q
+
\sum_{i = 1}^{N(\ep)} \| \Grad \bfphi_{\ep,i} \|_{L^q(B_{2r_i(\ep)}(x_i(\ep)); R^{d \times d})}^q.
\end{equation}
As $\bfphi$ is continuously differentiable, we get, on the one hand, 
\begin{equation} \label{A17}
\sum_{i = 1}^{N(\ep)} \| \Grad \bfphi \|_{L^q(B_{2r_i(\ep)}(x_i(\ep)); R^{d \times d})}^q \aleq \sum_{i = 1}^{N(\ep)} r_i(\ep)^3. 
\end{equation}	
On the other hand, by virtue of the uniform bounds \eqref{A13},
\[ 
\sum_{i = 1}^{N(\ep)} \| \Grad \bfphi_{\ep,i} \|_{L^q(B_{2r_i(\ep)}(x_i(\ep)); R^{d \times d})}^q \aleq 
\sum_{i=1}^{N(\ep)} \| \Grad {\bf E}_{\ep,i} [\bfphi] \|_{L^q(B_{2r_i(\ep)}(x_i(\ep)); R^{d \times d})}^q .
\]
As 
\begin{equation} \label{A18}
\Grad {\bf E}_{\ep,i} [\bfphi] \aleq |\Grad \bfphi | + \frac{1}{r_i(\ep)} |\bfphi| 
\end{equation}
we may infer that
\begin{equation} \label{A19}
\sum_{i = 1}^{N(\ep)} \| \Grad \bfphi_{\ep,i} \|_{L^q(B_{2r_i(\ep)}(x_i(\ep)); R^{d \times d})}^q \aleq 	
\sum_{i = 1}^{N(\ep)} r_i(\ep)^{3 - q}.
\end{equation}
In particular, for $q = 2$, hypothesis \eqref{A7} yields {\bf (M2)}. Finally, \eqref{A8} follows from \eqref{A19}.

\end{proof}

For $d = 2$, we have a similar result. 

\begin{Proposition} \label{PA2}
	Let 
\[
\Ome = \Omega \setminus K_\ep ,\ K_\ep \subset \Omega \subset R^2. 
\]
Suppose there exists a family of balls $B_{r_i(\ep)}(x_i(\ep))$ of radii $r_i(\ep) \in (0,\frac12)$, $i=1,\dots, N(\ep)$ such that 
\begin{align} 
K_\ep &\subset \bigcup_{i = 1}^{N(\ep)} B_{r_i(\ep)}(x_i(\ep)), \br {\rm dist}[ x_i , \partial \Omega] &> \sqrt{r_i(\ep)} , 
\ {\rm dist}[ x_i, x_j ] > \sqrt{r_i(\ep)} + \sqrt{r_j(\ep)} \ \mbox{for}\ i \ne j, \br 
\sum_{i=1}^{N(\ep)} |\log r_i(\ep)|^{-1} &\to 0 \ \mbox{as}\ \ep \to 0 .	\label{A20}
	\end{align}

Then {\bf (M2)} holds. Moreover, if
\begin{equation} \label{A21}
	h(\ep) = \min_{1\leq i \leq N(\ep)} r_i(\ep) |\log r_i(\ep)|,
	\end{equation}
then	 {\bf (M2bis)} is satisfied for all $\gamma>1$, specifically, \eqref{P10} holds for all $r<\infty$.
	\end{Proposition}
	
\begin{Remark}
	Compared to Proposition \ref{PA1}, the assumption $r_i(\ep)\in (0,\frac12)$ is made to ensure that a uniformly bounded Bogovski\u{\i} operator associated to the domain $B_{\sqrt{r_i(\ep)}}(x_i(\ep)) \setminus B_{r_i(\ep)} (x_i(\ep))$ exists, see \cite[Theorem~5.2]{DieRuzSch}.
\end{Remark}

\begin{proof}[Proof of Proposition \ref{PA2}]
	
	For $\bfphi \in C^1_c(\Omega; R^d)$, $\Div \bfphi = 0$, this time we consider the function 
\begin{align*} 
	\tilde{H}(Z)=\begin{cases}
	0 & {\rm if} \ 0\leq Z < r_i(\ep),\\
	2\bigg(1-\frac{\log Z}{\log r_i(\ep)}\bigg) & {\rm if} \ r_i(\ep)\leq Z < \sqrt{r_i(\ep)},\\
	1 & {\rm if} \ Z\geq \sqrt{r_i(\ep)},
	\end{cases}
\end{align*}
which we may convolute by an admissible convolution kernel to ensure $\tilde{H}\in C^\infty([0,\infty))$. Set now 
\begin{equation*}
\vc{E}_{\ep,i}[\bfphi] =  \bfphi(x) \tilde{H} \big(| x - x_i(\ep) |\big), 
\end{equation*}
and as before
\begin{equation*}
	\bfphi_{\ep,i} = \vc{E}_{\ep,i}[\bfphi] - \mathcal{B}_{\ep,i}[\Div \vc{E}_{\ep,i} [\bfphi]],
\end{equation*}
where $\mathcal{B}_{\ep,i}$ is the Bogovski\u{\i} operator associated to the domain
\[
Q_{i,\ep} = B_{\sqrt{r_i(\ep)}}(x_i(\ep)) \setminus B_{r_i(\ep)} (x_i(\ep))
\]
and enjoying the same properties \eqref{A12} and \eqref{A13}. Again, we have 
\begin{align*} 
	\bfphi_{\ep, i} &= 0 \ \mbox{in}\ B_{r_i(\ep)}(x_i(\ep)) \ , \br 
	\bfphi_{\ep, i} &= \bfphi \ \mbox{in}\ \Omega \setminus B_{\sqrt{r_i(\ep)}}(x_i(\ep)), \br 
	\Div \bfphi_{\ep, i} &= 0 \ \mbox{in}\ \Omega 
	\end{align*} 
for any $i = 1,\dots, N(\ep)$. Set finally
\begin{equation*}
	\bfphi_\ep (x) = \left\{ \begin{array}{l} \bfphi_{\ep,i} (x) \ \mbox{if}\ x \in B_{\sqrt{r_i(\ep)}}(x_i(\ep)), \\ 
		\bfphi \ \mbox{otherwise.} \end{array} \right.
\end{equation*}

By continuous differentiability of $\bfphi$, we have 
\begin{align*} 
\| \Grad \bfphi - \Grad \bfphi_\ep \|_{L^q(\Omega; R^{d \times d})}^q &\leq 
\sum_{i = 1}^{N(\ep)} \| \Grad \bfphi \|_{L^q(B_{\sqrt{r_i(\ep)}}(x_i(\ep)); R^{d \times d})}^q
+
\sum_{i = 1}^{N(\ep)} \| \Grad \bfphi_{\ep,i} \|_{L^q(B_{\sqrt{r_i(\ep)}}(x_i(\ep)); R^{d \times d})}^q\\
&\aleq \sum_{i = 1}^{N(\ep)} r_i(\ep) + \sum_{i = 1}^{N(\ep)} \| \Grad \bfphi_{\ep,i} \|_{L^q(B_{\sqrt{r_i(\ep)}}(x_i(\ep)); R^{d \times d})}^q.
\end{align*}
For the remaining term, we first calculate
\begin{equation*}
\Grad {\bf E}_{\ep,i} [\bfphi] \aleq |\Grad \bfphi | + \frac{1}{|x-x_i(\ep)| \, |\log r_i(\ep)|} |\bfphi| 
\end{equation*}
to see that
\begin{align*} 
\sum_{i = 1}^{N(\ep)} \| \Grad \bfphi_{\ep,i} \|_{L^q(B_{\sqrt{r_i(\ep)}}(x_i(\ep)); R^{d \times d})}^q &\aleq \sum_{i = 1}^{N(\ep)} \| \Grad \vc{E}_{\ep,i}[\bfphi] \|_{L^q(B_{\sqrt{r_i(\ep)}}(x_i(\ep)); R^{d \times d})}^q\\
&\aleq \sum_{i = 1}^{N(\ep)} \begin{cases}
|\log r_i(\ep)|^{-q} |r_i(\ep)^\frac{2-q}{2}-r_i(\ep)^{2-q}| & \text{if } q\neq 2,\\
|\log r_i(\ep)|^{-1} & \text{if } q=2.
\end{cases}
\end{align*}
In particular, for $q = 2$, hypothesis \eqref{A20} yields {\bf (M2)}. To show \eqref{A21}, note that for $q \neq 2$
\begin{align*}
\sum_{i = 1}^{N(\ep)} |\log r_i(\ep)|^{-q} |r_i(\ep)^\frac{2-q}{2}-r_i(\ep)^{2-q}| &\lesssim \bigg(\min_{1\leq i\leq N(\ep)} r_i(\ep) |\log r_i(\ep)|\bigg)^{-q} \bigg( \sum_{i = 1}^{N(\ep)} r_i(\ep)^{\min\big\{\frac{2+q}{2}, 2\big\}} \bigg) \\
&\lesssim \bigg(\min_{1\leq i\leq N(\ep)} r_i(\ep) |\log r_i(\ep)|\bigg)^{-q},
\end{align*}
and for $q=2$
\begin{align*}
\sum_{i = 1}^{N(\ep)} |\log r_i(\ep)|^{-1} &\leq \bigg(\min_{1\leq i\leq N(\ep)} r_i(\ep) |\log r_i(\ep)|\bigg)^{-2} \bigg( \sum_{i = 1}^{N(\ep)} r_i(\ep)^2 |\log r_i(\ep)| \bigg) \\
&\leq \bigg(\min_{1\leq i\leq N(\ep)} r_i(\ep) |\log r_i(\ep)|\bigg)^{-2} \bigg( \sum_{i = 1}^{N(\ep)} |\log r_i(\ep)|^{-1} \bigg) \\
&\leq \bigg(\min_{1\leq i\leq N(\ep)} r_i(\ep) |\log r_i(\ep)|\bigg)^{-2}.
\end{align*}

\end{proof}

Obviously, the above results are strongly related to problems in homogenization, see e.g. 
Allaire \cite{Allai4}, \cite{Allai3} or Tartar \cite{Tar1}.

\subsection{Domains containing thin frames}

In certain situations, hypothesis {\bf (M2)} may be difficult to verify because of the necessary solenoidality of the approximate sequence. We show that this condition can be relaxed provided the family $(\Omega_\ep)_{\ep > 0}$ are \emph{uniformly John domains}. Following Diening et al. \cite{DieRuzSch}, we say that $(\Omega_\ep)_{\ep > 0}$ is uniformly John of order $\alpha > 0$, if every pair of distinct points $x_1, x_2 \in \Omega_\ep$ can be connected by a rectifiable curve $\beta$ such that 
\[
{\rm cig}(\beta; \alpha) = \bigcup_{t \in [0, |\beta|] } \left\{ B \left( \beta(t); \frac{1}{\alpha} \min \left\{ t, |\beta|- t \right\} \right) \right\} \subset \Omega_\ep.
\] 

We claim the following result. 

\providecommand{\bfz}{\boldsymbol{\zeta}}
\begin{Proposition} \label{PPP1}
	
Let $\Omega \subset R^d$, $d \in \{2,3\}$ be a bounded domain and $K_\ep \subset \Omega$ a family of compact sets such that $(\Omega_\ep = \Omega \setminus K_\ep)_{\ep > 0}$ are uniformly John domains of order $\alpha$. Suppose the following holds:
	
\item {\bf(M2')} For any $\bfz \in C^1_c (\Omega; R^d)$, there exists a sequence 
$(\bfz_\ep)_{\ep > 0}$ with
\begin{equation} \label{PP8}
	\bfz_\ep \in W^{1,2}_0(\Omega_\ep ; R^d), \ \bfz_\ep \to \bfz \ \mbox{strongly in}\ W^{1,2}_0(\Omega; R^d).
\end{equation}

Then {\bf (M2)} holds.	
	
\end{Proposition}

\begin{proof}
	
Let $\bfphi \in C_c^1(\Omega; R^d)$ with $\Div \bfphi = 0$. The approximate sequence $(\bfphi_\ep)_{\ep>0}$ desired in {\bf (M2)} is constructed as the unique solution of the Stokes problem:
\begin{equation} \label{PP9}
	-\Del \bfphi_\ep + \Grad \Pi_\ep = -\Del \bfphi,\ \Div \bfphi_\ep = 0 \ \mbox{in}\ \Omega_\ep,\ \bfphi_\ep|_{\partial \Omega_\ep} = 0.  	
\end{equation}
Extending $\bfphi_\ep$ to be zero outside $\Ome$ we report the standard estimates 
\begin{equation*}
\| \bfphi_\ep \|_{W^{1,2}_0 (\Omega;R^d) } \aleq \| \Grad \bfphi \|_{L^2(\Omega;R^d)},\ 
\| \Grad \Pi_\ep \|_{W^{-1,2}(\Ome; R^d)} \aleq \| \Grad \bfphi \|_{L^2(\Omega; R^d)} \ \mbox{uniformly for}\ \ep \to 0.
\end{equation*}
Since the family $(\Ome)_{\ep > 0}$ is uniformly John, the Negative norm theorem \cite[Theorem 5.10]{DieRuzSch} yields
\[ 
\| \Pi_\ep \|_{L^2(\Ome)} \aleq  \| \Grad \bfphi \|_{L^2(\Omega;R^d)} \ \mbox{as long as we normalize the pressure,}
\ \intOe{ \Pi_\ep } = 0.
\]	
Consequently, at least for a suitable subsequence, 
\begin{equation} \label{PP12}
\bfphi_\ep \to \wbfphi \ \mbox{weakly in}\ W^{1,2}_0 (\Omega; R^d),\ 
\mathds{1}_{\Ome} \Pi_\ep \to \Pi \ \mbox{weakly in}\ L^2(\Omega).
\end{equation}
Next, let $\bfz \in C_c^1(\Omega; R^d)$. Using hypothesis {\bf (M2')}, we test equation \eqref{PP9} against the approximating sequence $\bfz_\ep \in W_0^{1,2}(\Ome; R^d)$ and send $\ep \to 0$ to obtain 
\begin{align*}
\intO{ \Grad \wbfphi : \Grad \bfz - \Pi \Div \bfz} &= \lim_{\ep \to 0} \intO{\Grad \bfphi_\ep : \Grad \bfz_\ep - \mathds{1}_{\Ome} \Pi_\ep \Div \bfz_\ep} \\
&= \lim_{\ep \to 0} \intO{\Grad \bfphi : \Grad \bfz_\ep} = \intO{\Grad \bfphi : \Grad \bfz}.
\end{align*}
Since $\bfz$ was arbitrary, this yields
\begin{equation} \label{PP11}
-\Del \wbfphi + \Grad \Pi = -\Del \bfphi,\ \Div \wbfphi = 0 \ \mbox{in}\ \Omega,\ 
\bfphi|_{\partial \Omega} = 0.	
\end{equation}
Thus, by uniqueness of solutions to the Stokes problem \eqref{PP11}, we conclude 
\[
\bfphi = \wbfphi \ \mbox{as}\ \Div\bfphi = 0 \ \mbox{in}\ \Omega.
\]

Finally, the relations \eqref{PP9}--\eqref{PP11} imply 
\[
\intO{ |\Grad \wbfphi_\ep |^2 } \to  \intO{ |\Grad \wbfphi |^2 } = \intO{ |\Grad \bfphi |^2 },
\]
which yields the strong convergence required in {\bf (M2)}.
	
\end{proof}

As an application, we may consider a bounded domain $\Omega \subset R^3$ and a \emph{finite} number of compact non-intersecting curves $\{ \beta_i \}_{i=1}^N$, $\beta_i \subset \Omega$, $i=1,\dots,N$, together with a family of compact sets 
\[
K_\ep = \bigcup_{i=1}^N \Ov{ \mathcal{U}_\ep \left[ \beta_i \right] },\ \mathcal{U}_\ep [\beta_i] = \left\{ x \in \Omega \ \Big|\ {\rm dist}[x, \beta_i] < \ep \right\}.
\]  
It is easy to check that the family $\Omega_\ep = \Omega \setminus K_\ep$ are uniformly John domains (for small $\ep > 0$) as they satisfy the uniform cone property. Moreover, as the Newtonian capacity of a compact curve in $R^3$ is zero, hypothesis {\bf (M2')} holds together with the conclusion of Proposition \ref{PPP1}.

\begin{Remark} \label{RSR1}
	
The result can be extended to a family of $N(\ep)$ curves, $N(\ep) \to \infty$, under suitable restrictions imposed on the geometry of the curves and their number $N$ as a function of $\ep$. In particular, the approach developed by Namlyeyeva et al. \cite{NamNecSkr} can be adapted to handle the case of straight lines (cylinders). We leave the details to the interested reader.
	
\end{Remark}



\def\cprime{$'$} \def\ocirc#1{\ifmmode\setbox0=\hbox{$#1$}\dimen0=\ht0
	\advance\dimen0 by1pt\rlap{\hbox to\wd0{\hss\raise\dimen0
			\hbox{\hskip.2em$\scriptscriptstyle\circ$}\hss}}#1\else {\accent"17 #1}\fi}

\end{document}